\documentclass[12pt,letterpaper]{article}

\usepackage{diagbox}
\usepackage{mathtools}
\usepackage{bbm}
\usepackage{ulem}
\usepackage{latexsym}
\usepackage{epsfig}
\usepackage{amsmath,amsthm,amssymb,enumerate}
\usepackage[utf8]{inputenc}
\usepackage[a-1b]{pdfx}
\usepackage{hyperref}
\hypersetup{colorlinks,linkcolor=black,filecolor=black,urlcolor=black,citecolor=black}

\allowdisplaybreaks
\usepackage{pdftricks}
\begin{psinputs}
  \usepackage{pstricks}
\end{psinputs}

\usepackage{color}

\newcommand{\ep}{\varepsilon}

\newcommand{\TSP}{\mathrm{TSP}}
  \newcommand{\MST}{\mathrm{MST}}

\def\cH{{\mathcal H}}
\def\nn{\nonumber}
\def\a{\alpha} \def\b{\beta} \def\d{\delta} 
\def\e{\varepsilon}    \def\g{\gamma}
\def\G{\Gamma}  \def\k{\kappa}
\def\z{\zeta}     \def\l{\lambda}
  \def\n{\nu} \def\p{\pi}
   
\def\t{\tau}

\def\cP{{\cal P}}

\def\bx{{\bf x}}
\def\by{{\bf y}}

\newtheorem{theorem}{Theorem}
\newtheorem{lemma}[theorem]{Lemma}
\newtheorem{observation}[theorem]{Observation}

\newtheorem{claim}{Claim}

\newtheorem{proposition}[theorem]{Proposition}

\newcommand{\brac}[1]{\left(#1\right)}

\newcommand{\bfrac}[2]{\left(\frac{#1}{#2}\right)}

\def\cE{{\cal E}}

\newcommand{\set}[1]{\left\{#1\right\}}
\def\sm{\setminus}

\def\E{\mathbb{E}}

\def\Pr{\mathbb{P}}

\newcommand{\ignore}[1]{}

\def\cA{{\mathcal A}}

\def\cE{{\mathcal E}}

\def\cG{{\mathcal G}}
\def\cH{{\mathcal H}}

\def\cL{{\mathcal L}}

\def\cP{{\mathcal P}}

\def\cS{{\mathcal S}}

\def\cX{{\mathcal X}}
\def\cY{{\mathcal Y}}

\newcommand{\beq}[2]{\begin{equation}\label{#1}#2\end{equation}}
\newcommand{\mults}[1]{\begin{multline*}#1\end{multline*}}

\def\nn{\nonumber}

\def\cG{\mathcal{G}}

\usepackage{tikz}
\usetikzlibrary{decorations.pathmorphing}
\usetikzlibrary{positioning}
\usetikzlibrary{arrows,automata}
\usetikzlibrary{shapes.misc}
\usetikzlibrary{backgrounds}
\usetikzlibrary{arrows,shapes}

\begin{document}
\title{On Minimum Cost Rainbow Structures}
\author{Patrick Bennett\thanks{Department of Mathematics, Western Michigan University, Kalamazoo MI 49008-5248, Research supported in part by Simons Foundation Grant \#848648.}\and Quentin Dubroff\thanks{Department of Mathematical Sciences, Carnegie Mellon University, Pittsburgh PA15213.}\and Alan Frieze\thanks{Department of Mathematical Sciences, Carnegie Mellon University, Pittsburgh PA15213, Research supported in part by NSF grant DMS2341774.}\and Wesley Pegden\thanks{Department of Mathematical Sciences, Carnegie Mellon University, Pittsburgh PA15213, Research supported in part by NSF grant DMS1700365.}}
\maketitle
\date{}
\begin{abstract}
We discuss the expected minimum cost of rainbow  spanning trees and Hamilton cycles in randomly edge colored random graphs. 
\end{abstract}
\section{Introduction}
The minimum length of structures like spanning trees and Hamilton cycles among random points in Euclidean space has been a topic of interest for the better part of a century.  Beardwood, Halton and Hammersley \cite{BHH} showed that the minimum cost of a tour through $n$ random points in the unit square is asymptotic to $\b_{TSP} n^{1/2}$ with high probability (we say an event $\cE_n$ occurs {\em with high probability} or w.h.p.\ if $\Pr(\cE_n) \rightarrow 1$ as $n \rightarrow \infty$). Here $\b_{TSP}$ is a positive constant, still unknown after more than 60 years. Steele \cite{Steele1} generalised this result to {\em Euclidean Functionals} and showed that the minimum cost of a spanning tree on $n$ random points is asymptotic to $\b_{MST} n^{1/2}$, w.h.p. In both cases, the cost of edges is determined by the Euclidean distance between points. Steele's work shows that there are many instances of optimisation problems where the expected minimum cost grows like $\b n^{1/2}$ for unknown constants $\b$. Frieze and Pegden \cite{FP1} showed the constants for lengths of lower bound structures for the tour (like the minimum spanning tree, or twice a matching) really are distinct from the constant for the TSP, even though the best known explicit bounds on the constants generally overlap.

When the costs are generated independently, the situation becomes much clearer. In the case of spanning trees where the edges of the complete graph $K_n$ are given independent uniform $[0,1]$ random costs, Frieze \cite{F1} showed that $\E(L_n)\sim\z(3)=\sum_{k=1}^\infty\frac{1}{k^3}$, where $L_n$ denotes the (random) minimum cost of a spanning tree. 

The expected cost of the minimum cost Hamilton cycle in the complete graph $K_{n}$ is asymptotic to $\t=\frac12\int_{x=0}^\infty y(x)dx$ where $y$ is the positive solution to $\brac{1+\frac{x}{2}}e^{-x}+\brac{1+\frac{y}{2}}e^{-y}=1$. This was proved by Wastl\"und \cite{Wast1}. Frieze \cite{F2} showed that the expected minimum cost of a Hamilton cycle is asymptotically equal to  the expected minimum cost of a 2-factor. 

If we orient the edges then it follows from the work of Karp \cite{K1} and Aldous \cite{A1} that the expected minimum cost of a directed Hamilton cycle is asymptotic to $\z(2)=\sum_{k=1}^\infty\frac{1}{k^2}=\frac{\p^2}{6}$.

In this work we add a spot of color to the problems. We assume that the edges of the graphs in question have been given uniformly random colors and ask for minimum cost rainbow Hamilton cycles or spanning trees. In this context, a set $S$ of edges is {\em rainbow} colored if all the colors in $S$ are distinct. We do this within two contexts: either we have $n$ random points in $[0,1]^2$ and the cost of an edge is the Euclidean distance between its endpoints, or else we have independent uniform $[0,1]$ random costs.

In the first context we let $\cX=\set{\bx_1,\bx_2,\ldots,\bx_n}$ where the $\bx_i$ are independently chosen uniformly from the unit square $[0,1]^2$. The cost of edge $\set{\bx_i,\bx_j}$ is the Euclidean length $|\bx_i-\bx_j|$. 
Our first results concern the length $Z_{\MST}$ of the minimum rainbow spanning tree in this model,  when we randomly color labeled from a set of $n-1$ colors.  (Note that with fewer colors, no rainbow spanning tree would exist. Also, even if we have enough colors there is always some small probability that there is no rainbow spanning tree at all, in which case $Z_{\MST}$ is undefined.)  We let $Z^*_{\MST}$ denote the length of the minimum spanning tree (without the rainbow constraint); recall that $Z^*_{\MST}\sim \beta_{\MST} \sqrt n$ w.h.p., for some constant $\beta_{\MST}$. 

\begin{theorem}\label{EucTrees}
 In the Euclidean setting, with edges labeled randomly by $n-1$ colors, we have $Z_{\MST} = \Theta(\sqrt n)$ w.h.p.\ and $Z_{\MST}-Z^*_{\MST}=\Theta(\sqrt{n})$ w.h.p.
\end{theorem}

In particular, Theorem \ref{EucTrees} says that the minimum length doesn't change by more than a constant factor under a rainbow constraint, but that a constant-factor change does necessarily occur.  To put Theorem \ref{EucTrees} in context, there is an important sense in which it is much less obvious that the minimum rainbow spanning tree should have length $\Theta(\sqrt{n})$ than it is for the various optimization problems (matchings, 2-factors, linear programming relaxations of the TSP, etc) to which the general methods derived from Beardwood, Halton, and Hammersley and from Steele have been applied.  In particular, the rainbow condition seems to preclude a straightforward application of subadditivity, which explains, for example, why we do not yet know that $Z_{\MST}/\sqrt{n}$ has a limit.  In some sense, the success of subadditivity approaches to structure lengths in random point-sets reflects the sense in which the asymptotic structure lengths really depend only on the local geometry of the ambient space.  (For example, in the Beardwood-Halton-Hammersley theorem, the unit square can be replaced by any bounded open set of $\mathbb{R}^2$ of measure 1, and the asymptotic length of a tour through $n$ points remains the same, with the same unknown constant $\beta_{\TSP}$. But the rainbow condition we impose is a global constraint on the spanning tree that cannot be easily decomposed into noninteracting parts.

For the minimum length $Z_{\TSP}$ of a rainbow Hamilton cycle in the Euclidean model, we do not achieve an optimal result up to constant factors unless we allow a slight excess in the number of available colors.  

\begin{theorem}\label{EucTours}
In the Euclidean setting, with edges labeled randomly by $q=(1+\ep)n$ colors, we have (a) $Z_{\TSP}=O(n^{1/2}\log n)$, and (b) $Z_{\TSP}-Z_{\TSP}^*=\Omega(\sqrt{n})$ w.h.p. Indeed, (a) holds for any set of $n$ points with randomly colored edges.
\end{theorem}
Here like before the starred version $Z^*_{\TSP}$ represents the minimum length in the absence of the rainbow condition.

In the second scenario, $[0,1]^2$ is replaced by the complete graph $K_n$ and the costs $C(i,j)$ are independent uniform $[0,1]$ random variables.  We prove a general theorem (see Theorem \ref{rweight}) for cheap rainbow substructures in randomly colored structures with random costs. The proof is an adaptation of some of the recent work on thresholds in spread hypergraphs (see \cite{BFM}) based on Frankston, Kahn, Narayanan, and Park's recent proof of the fractional Kahn--Kalai conjecture \cite{FKNP}. Our Theorem \ref{rweight} implies the following:
\begin{theorem}\label{th2} 
The minimum cost solution of the following problems have value $O(1)$ w.h.p. The proofs are non-constructive.
\begin{enumerate}[(a)]
\item Traveling Salesperson Problem in $K_n$ with at least $n$ colors for the edges. 
\item Minimum cost perfect matching in complete $k$-uniform hypergraphs, $k\geq 2$ with at least $n/k$ colors for the edges.
\item Minimum cost classes of degree bounded spanning trees in $K_n$ with at least $n-1$ colors for the edges e.g. spanning binary trees.
\end{enumerate}
\end{theorem}
The rest of the paper is organized as follows. In Section \ref{sec:treeUB} we prove the first assertion of Theorem \ref{EucTrees}, i.e.\ that $Z_{\MST} = \Theta(\sqrt{n})$ w.h.p. In Section \ref{sec:tours} we prove the first assertion of Theorem \ref{EucTours}, i.e.\ that $Z_{\TSP}=O(\sqrt {n}\log n)$ for $n$ colors. In Section \ref{sec:separate} we prove the remaining assertions in Theorems \ref{EucTrees} and \ref{EucTours}, i.e.\ that $Z_{\MST}-Z^*_{\MST}=\Theta(\sqrt{n})$ for $n$ colors and that $Z_{\TSP}-Z_{\TSP}^*=\Theta(\sqrt{n})$ for $(1+\e)n$ colors. In Section \ref{sec:spread} we prove Theorem \ref{th2}. Finally, we make some observations and state open questions in Section \ref{sec:conclusion}.

\section{Finding a rainbow spanning tree with $n-1$ available colors}\label{sec:treeUB}
In this section, we describe a simple polynomial time algorithm for finding a rainbow spanning tree of cost $O(n^{1/2})$. (We realise that finding a cheapest rainbow tree is solvable in polynomial time, see for example Lawler \cite{Law}. But the weighted matroid intersection algorithm is not so easy to analyse from a probabilistic point of view.) For $\bx=(x,y)\in \cX$ we let $\xi_1(\bx)=x,\xi_2(\bx)=y$. Assuming that $\bx_1=(x_1,y_1),\bx_2=(x_2,y_2),\ldots,\bx_{n}=(x_n,y_n)$ are ordered in increasing value of $\xi_2$, let $\cX=\set{\bx_1,\bx_2,\ldots,\bx_{n}}$ and $\cX_1=\set{\bx_1,\bx_2,\ldots,\bx_{n - 1}}$.

We define a weighted bipartite (multi-)graph $\G$ with bipartition $\cX_1 \cup C$ where $C=[n-1]$ is our set of colors. The edge set will be the union $E_1\cup E_2$ of sets which we define next; for both sets $E_1$ and $E_2$, an edge $(\bx,c)$ will correspond to some edge $(\bx,\bx')$ with $\xi_2(\bx) < \xi_2(\bx')$ in the embedded graph that has been assigned color $c$. This last property guarantees that a perfect matching in $\G$ corresponds to a rainbow spanning tree in the embedded graph. So to prove the first assertion in Theorem~\ref{EucTrees}, it suffices to show that w.h.p.
\begin{equation}\label{match}
\text{$\G$ contains a perfect matching,}
\end{equation}
and
\begin{equation}\label{wt}
\text{the total weight of all edges in $\G$ is $O(n^{1/2})$.}
\end{equation}

First we define  the set of directed edges
\[
E^+=\{(\bx,\by)\in \cX_1\times \cX:\; \xi_2(\by)>\xi_2(\bx)\}.
\]
For $\bx\in \cX$, we let $N(\bx)=\set{\by:(\bx,\by)\in E^+}$. We let $K$ be a constant large enough to make the argument work and define $E_1$ as follows. For each $\bx\in \cX_1,\,c\in C$ there is an edge $e_{\bx,c}=\set{\bx,c}$ in $E_1$ if $c$ is one of the colors of the $K$ shortest edges in $N(\bx)$ (noting that if $i > n - K$, then we include all edges $e_{\bx,c}$ where $c$ is some color of an edge in $N(\bx_i)$). We will prove that the total weight of the edges in $E_1$ is $O(n^{1/2})$ w.h.p.\ at the end of this section.

In addition, there is another set of edges $E_2$. We first define a set $A\subseteq \cX_1$. For a large positive value $B$ and $1\leq j\leq L=\log^2n$, we let 
\[
A_j=\set{\bx\in \cX_1:\exists \by\ \text{s.t.} (\bx,\by)\in E^+\text{ and }|\bx-\by|\in \left[\frac{Bj^2}{\sqrt{n}}, \frac{B(j+1)^2}{\sqrt{n}}\right)}.
\]
We then let $A=\bigcap_{j=1}^{L}A_j$. We will show that $A$ is most of $\cX$, but first, we finish defining the edge set $E_2$. 

Define a set $E_A$ by choosing, for each $\bx \in A$ and $j \in [L]$, one edge  $(\bx,\by)\in E^+$ incident to $\bx$ of length in the interval
\begin{equation}
  \label{interval}
        [Bj^2/\sqrt{n},B(j+1)^2/\sqrt{n}).
\end{equation}
The \emph{level} of $e \in E_A$ is the unique $j$ such that the length of $e$ lies in $[Bj^2/\sqrt{n},B(j+1)^2/\sqrt{n})$. We obtain $E_2$ by adding, for each color $c$, the edges $(\bx, c)$, where $\bx$ is a lower endpoint of one of the $K$ lowest level edges of color $c$ in $E_A$; here, we break ties by choosing a random ordering of the edges of each fixed level (and this random ordering is chosen independently for each distinct color). Note that if 
$|A|$, is large, as we will now show, then our choice of $L = \log^2 n$ guarantees that each color $c$ has $K$ incident edges in $E_2$ w.h.p.

We have that for any $j\in[L]$, if $\xi_2(\bx) < 1 - n^{-2/5}$, then the probability that $\bx\notin A_j$ is at most 
\beq{Aeq}{
\brac{1 - \frac{2\p B^2j^3}{n}}^{n-1} \leq e^{-6B^2j^3}.
}
So, 
\[\Pr(\bx \not\in A) \leq  \Pr(\bx \not\in A \,|\, \xi_2(\bx) < 1 - n^{-2/5}) + n^{-2/5} \leq \sum_j e^{-6B^2j^3} + n^{-2/5}.\]
For large $B$ (and $n$), we find
\[\Pr(\bx \not \in A) < e^{-5B^2}.\]
So
\beq{sizeA}{
\E(|A|)\geq n - e^{-5B^2}n = n(1-e^{-5B^2}).
}
We now use an inequality of Warnke to show that $|A|$ is large with very high probability.
\begin{lemma}[Warnke~\cite{War}]\label{warnke}
Let $W=(W_1,W_2,\ldots,W_n)$ be a family of independent random variables with $W_i$ taking values in a set $\Lambda_i$. Let $\Omega=\prod_{i\in[n]}\Lambda_i$ and suppose that $\cG \subseteq \Omega$ and $f:\Omega\to{\mathbb R}$ are given. Suppose also that whenever $\bx,\bx'\in \Omega$ differ only in the $i$-th coordinate 
\[
|f(\bx)-f(\bx')|\leq \begin{cases}c_i&if\ \bx\in\cG.\\d_i&otherwise.\end{cases}
\]
If $Y=f(X)$, then for all reals $\gamma_i>0$,
\begin{equation}\label{Wineq}
\Pr(Y\geq \E(Y)+t)\leq \exp\set{-\frac{t^2}{2\sum_{i\in[n]}(c_i+\gamma_i(d_i-c_i))^2}}+\Pr(W\notin \cG)\sum_{i\in [n]}\gamma_i^{-1}.
\end{equation}
\end{lemma}
To use this lemma we need to define a ``good'' event $\cG$. So, we let $\cG$ be the event that  for all \bx,
\[
\n_\bx=|\set{\by:|\bx-\by|\leq 2B\log^{ 4}n/n^{1/2}}|  < 15B^2 \log^{ 8}n.
\]
The Chernoff bounds imply that 
\begin{equation}\label{negG}
\Pr(\neg\cG)=e^{-\Omega(\log^8n)}.
\end{equation}
To apply Lemma \ref{warnke} we let $W_i=\bx_i$ and $f(\cX)=n - |A|$ and take $c_i={ 30B^2 }\log^{ 8}n$ and $d_i=n$ for $i\in[n]$. Putting $\g_i=n^{-3}$ and our bound on $\Pr(\neg\cG)$ in \eqref{Wineq}, we see that $|A| > \E(|A|) -  n^{2/3}$ w.h.p. We conclude that w.h.p.
\begin{equation}\label{Asize}
    |A| > (1 - e^{-5B^2})n.
\end{equation}

\bigskip

The purpose of this construction of the edge set $E_2$ is to guarantee the following property. 
\begin{observation}
\label{symmetry}  The (multi)-set of neighbors $\bx$ for a color $c$ in $E_2$ is symmetric with respect to permutation of labels of $A$.
\end{observation}
\begin{proof}
  This follows from construction of the set $E_2$, since each vertex $\bx$ is the lower endpoint of exactly one edge in $E_A(x)$ for each interval length.
\end{proof}
To summarize, letting $N_1(\cdot)$ and $N_2(\cdot)$ refer to neighborhoods using edges from $E_1$ and $E_2$ respectively, the neighborhood sets $N_1(\bx_i)$ ($i \in [n-1]$) are a collection of independent uniform random (multi)sets of $C$ (of size $\min\{K, n - i\}$). And the edge (multi)sets $N_2(c)$ ($c\in C$) are uniform in the sense of Observation~\ref{symmetry}, and we will show in Claim~\ref{clm:1} that they exhibit some approximate correlation in our favor.
 
 We now show using Hall's theorem that \eqref{match} holds. We first show that $S\subseteq \cX_1$ implies that $|N(S)|\geq |N_1(S)|\geq |S|$ for $|S|\leq n_0:=e^{-4/K}(n-1)$. We do this by finding a perfect matching between $\cX' = \{x_{n- K},\dots, x_{n-1}\}$ and a set $Y \subset C$  and subsequently showing that for every $S$ with $ S \cap \cX' = \emptyset$ and $|S| \leq n_0$, we have $|N_1(S)\setminus Y| \geq |S|$. Finding a matching covering $\cX'$ just requires checking Hall's condition on $\cX'$: each $S \subseteq \cX'$ of size $s$ has at least ${s+1 \choose 2}$ outgoing $E_1$ edges, so the probability there is a set $S\subseteq \cX'$ violating Hall's condition is at most
  \[\sum_{s = 1}^K {K \choose s}{|C| \choose s-1}\left(\frac{s-1}{|C|}\right)^{s+1 \choose 2} \leq 2^K \sum_{s = 2}^K \left(\frac{e|C|}{s-1}\left[\frac{s-1}{|C|}\right]^{(s+1)/2}\right)^{s},\]
  which is $O(1/n)$. Thus Hall's theorem implies the existence of a set $Y \subseteq C$ with a perfect matching to $\cX'$. For $S \subset \cX_1\setminus \cX'$, let $n_1 = n - K - 1$, note
\begin{align*}
\Pr(\exists S\subseteq \cX_1\setminus \cX',|S|\leq n_0,|N_1(S)\setminus Y|\leq |S|)&\leq \sum_{s=K}^{n_0}\binom{n_1}{s}\binom{n-1}{s}\bfrac{s+K}{n-1}^{Ks}\\
&\leq \sum_{s=K}^{n_0}\bfrac{(n-1)e}{s}^{2s}\bfrac{s+K}{n-1}^{Ks}\\
& \leq e^{K^2} \sum_{s=K}^{n_0} \left[e^2 \brac{\frac{s}{n-1}}^{(K-2)}\right]^s\\
& =O(n^{-(K-3)}),
\end{align*}
where in the second-to-last line we used $(1+K/s)^{Ks} \le e^{K^2}.$ 

 To check Halls condition for $S$ with $|S|>n_0$, we look at the sizes of the $E_2$ neighborhoods of subsets of $C$. If $S\subseteq \cX$ satisfies ${|N(S)|}\leq|S|-1$ then for $T=C\sm N(S)$ we have 
\beq{4}{
|N_2(T)|\leq |N(T)|\leq n - 1-|S|\leq n- 1 -|N(S)|-1=|T|- 1.
}
Note also that 
\begin{equation}\label{n2} |T|\leq  n-1-n_0\leq n_2:= 4(n-1)/K.\end{equation}

Let the positions of the points $\cX$ be denoted by $\cP$. Assume now that $\cP$ is given and that $|A|\geq (1-e^{-5B^2})n$. 
\begin{claim}\label{clm:1} 
Let $\cE_c=\cE_c(S)$ be the event that color $c$'s $E_2$-neighbors are in some set $S\subseteq A$. Then for any fixed $S,T,P,$ 
\[
\Pr\brac{\bigcap_{c\in T}\cE_c(S)\;\bigg|\;\cP=P}\leq e^{9|T|}\prod_{c\in T}\Pr(\cE_c(S)\mid\cP=P).
\]
\end{claim}
\begin{proof}
  If $T=\set{c_1,c_2,\ldots,c_\tau}$ then
\[
\Pr\brac{\bigcap_{c\in T}\cE_c \bigg| \cP=P}=\prod_{j=1}^\tau\Pr\brac{\cE_{c_j}\ \bigg| \bigcap_{k<j}\cE_{c_{k}},\cP=P}.
\]
Now write
\[
\Pr\brac{\cE_{c_j}\ \bigg| \bigcap_{k<j}\cE_{c_{k}},\cP=P}=\sum_{\cH\in H}\Pr(\cE_{c_j}\mid\cH,\cP=P)\Pr\brac{\cH\mid \cap_{k < j}\cE_{c_k},\cP=P},
\]
where each \emph{history} $\cH\in H$ is an event (defined following \eqref{PH}) consisting of the choices for the $\leq K$ lowest level edges in each of the colors $c_1,\dots,c_{j-1}$ (under each color's collection of random orderings of edges of each fixed level). 

To prove our claim it suffices to show that 
for all choices of the event $\cH\in H$,
\beq{PH}{
\Pr(\cE_{c_j}\mid \cH,\cP=P)\leq e^9\Pr(\cE_{c_j}\mid\cP=P)
}
 For this, we write $\cH = \cap_{i < j} \cH_i$, where $\cH_i$ is color $i$'s random orderings of the edges of each fixed level as well as its first $K$ edges in $E_A$ under this ordering (which are all edges with lower endpoint in $S$). Note that each $\cH_i$ introduces a set $G_i$ of $K$ edges in $E_A$ which are colored by $c_i$ and an additional set $F_i$ of (usually linearly many) edges that may not receive the color $c_i$. We now let $L=(\bx_i,\by_i),i=1,2,\ldots,|{E_A}|$ be a listing of the edges in $E_A$ in increasing order of level, where edges of the same level are ordered uniformly at random, and we consider some history $\cH_j$ that determines the event $\cE_{c_j}$.

The probability $\Pr(\cH_j\mid \cH,\cP=P)$ is determined by the outcome of an experiment $\Lambda_\cH=\Lambda_\cH(P)$ (the \emph{conditioned experiment}) in which we consider the edges in $L$ one at a time and choose its color according to the constraints imposed by the event $\cH$: when we reach the $t^{th}$ edge, if $t \in G_i$ for some $i < j$, the edge is colored by $c_i$, and otherwise the $t^{th}$ edge is uniform over $C\sm I_t$, where $I_t:=\{c_i : t \in F_i\}$.  We compare this conditioned experiment to the \emph{auxiliary experiment} $\Lambda=\Lambda(P)$ in which we simply randomly assign random colors from $C$ to edges in $E_A$; note that 
\[\sum_{\cH_j} \Pr(\cH_j\mid \cP = P) = \Pr(\cE_{c_j}\mid\cP=P),\]
so for \eqref{PH}, it suffices to show
\begin{equation}\label{ptwise}\Pr(\cH_j\mid \cH,\cP=P) \leq e^{9}\Pr(\cH_j\mid \cP=P)\end{equation}
for every $\cH_j$.

Defining $p = (n-1)^{-1}$, we have 
\[\Pr(\cH_j\mid \cP=P) =  p^{K}(1 - p)^{|F_j|}.\]
We similarly calculate $\Pr(\cH_j\mid \cH,\cP=P)$, the probability in the conditioned experiment, but now we have parameters $a_t = |I_t|$ and $p_t$, the probability that the $t$th edge receives color $c_j$ in the conditioned experiment. The important fact is that $p_t = 0$ if $t \in \cup_{i < j} G_i$ and otherwise
\[p_t = (n - 1 - a_t)^{-1} \leq (n - 1 - n_2)^{-1} < e^{4.5/K} p\]
where the last inequality uses our definition of $n_2$ in \eqref{n2} and the fact that $K$ is large.
We calculate 
\[\Pr(\cH_j\mid \cH,\cP=P) = \prod_{t \in G_j} p_t \prod_{t \in F_j} (1 - p_t)\leq e^{4.5}p^K\prod_{t \in F_j} (1 - p_t).\]
We finish the proof by noting that $1 - p_t \leq 1- p$ if $t \not\in \cup_{i < j} G_j$, and if $t \in \cup_{i < j} G_j$ then $1 - p_t = 1$ on this set of size $(j-1)K < n_2K$, so
\[\prod_{t \in F_j} (1 - p_t)\leq (1 - p)^{|F_j|}(1 - p)^{-n_2K} < (1 - p)^{|F_j|}e^{4.5}\]
for large enough $K$ and $n$.

This implies \eqref{ptwise} and completes the proof of Claim \ref{clm:1}.
\end{proof}

\begin{claim}\label{clm:3}
    If the event $\cG$ holds (where $\cG$ is defined above \eqref{negG}), then for any fixed $S \subseteq A$ of size $s$, (under the probability of choosing random colors for each edge)
\[
 \Pr(\cE_c(S)\mid A) \leq \bfrac{s}{|A|}^{\sqrt{K}} + O(n^{-\sqrt{K} + 2}\log^{8\sqrt{K}} n).
\]
 \end{claim}
 \begin{proof}
    Fix $\bx\in A$ and recall the definition of $\nu_\bx$ above \eqref{negG}. Let $\cA$ be the event that for each $\bx \in A$, each color appears fewer than $\sqrt{K}$ times on edges of length at most $2Bn^{-1/2}\log^4 n$. The probability of this event failing for fixed $\bx \in A$ is at most
    \[|C| {\nu_\bx \choose \sqrt{K}}|C|^{-\sqrt{K}},\]
    and taking a union bound over $A$ and using \eqref{negG} gives the second bound in the claim. 
    
    So we may assume that $\cA$ holds. Since all edges of $E_2$ have length at most $2Bn^{-1/2}\log^4 n$, this implies that each $c \in C$ has at least $\sqrt{K}$ neighbors in $E_2$.
   As the distribution of the (multi)-set of neighbors along $E_2$ edges of $c$ is symmetric with respect to permutation of the vertex set (Observation \ref{symmetry}), we have that
   \[
   \Pr(\cE_c(S)\mid A, \cA) \leq \frac{\binom{s}{\sqrt{K}}}{\binom{|A|}{\sqrt{K}}}\leq  \bfrac{s}{|A|}^{\sqrt{K}}. 
   \]

 \end{proof}

Note that by \eqref{Asize}, if $K$ is large, then the bound in Claim~\ref{clm:3} implies
\[\Pr(\cE_c(S)\mid A) \leq \bfrac{2s}{n}^{\sqrt{K/2}}.\]
We can now finish the proof that $\G$ has a perfect matching. Recalling \eqref{4}, we just need to show that $|N_2(T)| \geq |T|$ for every $T \subseteq C$ with $|T| \leq n_2$. But the probability of this event is at most
\[
\Pr(\exists S\subseteq A,T\subseteq C,|S|=|T|=t: N_2(T)\subseteq S\mid\cP=P)\leq \binom{n}{t}^2\bfrac{2e^{9}t}{n}^{t\sqrt{K/2}}.
\]
Summing over $t$ (and recalling $n_2 = 4(n-1)/K$ for a large constant $K$), we get
\[
\Pr(\exists S,T,|S|=|T| \le n_2: N_2(T)\subseteq S)\leq o(1)+\sum_{t=1}^{n_2}\binom{n}{t}^2\bfrac{2e^9t}{n}^{t\sqrt{K/2}}=o(1).
\]
This completes the proof that $\G$ has a perfect matching $M$ w.h.p. This matching $M$ gives us an acyclic rainbow set of $n-1$ edges and so defines a rainbow spanning tree. 

We now crudely bound the cost of $M$ by the cost of all the edges in $\G$. We begin with the cost of $E_2$. We show that under our assumption that $|A| > n/2$ (see \eqref{Asize}) the expected cost of the edges in $E_2$ is $O(n^{1/2})$. Moreover, given $A$, the cost of edges from $E_2$ is an $O(n^{-1/2}\log^4 n)$-Lipschitz function of the colors of $E_A$ (the Lipschitz constant coming from our definition of $E_A$), so we can convert the bound on the expectation to a probability bound using Warnke's inequality, Lemma~\ref{warnke}. To bound the expected cost given $A$, we fix a color $c$ and define $Y$ as the index of the $K$th lowest index edge of color $c$ in $E_A$. Since edges of index $j$ have weight $O(j^2 n^{-1/2})$, it suffices to show that
\begin{equation}\label{Yj}\Pr(Y > j) < \exp[-\Omega(j)].\end{equation}
The event $Y > j$ means that color $c$ appeared fewer than $K$ times among the edges of $E_A$ with index at most $j$. And under our assumption that $|A| > n/2$, this is at most the probability that $X < K$ where $X$ is distributed as $\text{Binomial}(jn/2,1/n)$. If $j> 4K$, the Chernoff bounds immediately give \eqref{Yj} (and for smaller $j$, we may take the implicit constant in \eqref{Yj} small enough to make the inequality trivial), completing the proof.

We now turn to the cost of the edges in $E_{1}$. First, if $i > n - 2n^{-1/2}$, then we use the trivial bound that the $K$ shortest edges incident to $\bx_i$ each have length at most $1$; so the total length of the $E_1$ edges incident to $\bx_i$ with $i > n - 2n^{-1/2}$ is at most $2Kn^{1/2}$. For smaller $i$, we define $Z = Z_i$ as the expected value of the $K$th shortest edge incident to $\bx_i$ and show
\[\E Z = O(n^{-1/2}).\]
(As in the previous case of $E_2$ edges, we can sum this bound over $i \leq n - 2n^{1/2}$ to get a bound on the expected cost of $E_1$ edges incident to $\bx_i$'s with $i \leq n - 2n^{1/2}$, and this expectation bound can be converted to a probability bound using Lemma~\ref{warnke}.)
For this bound, we show that for large enough $D_0$, if $D_0 < D < n^{1/2}$, then
\begin{equation}\label{PZ}\Pr(Z > Dn^{-1/2}) < \exp[-D/8].\end{equation}
We prove this bound for $i = n - 2n^{1/2}$ noting that the argument is the same or easier for smaller $i$. 
First, Chernoff bounds imply that $\xi_2(\bx_i) \in [1 - 4n^{-1/2}, 1 - n^{-1/2}]$ with probability at most $\exp[-n^{1/2}/4]$. Now, suppose $\xi_2(\bx_i) = y \in [1 - 4n^{-1/2}, 1 - n^{-1/2}]$ and (without loss of generality) that $x:=\xi_1(\bx_i) \leq 1/2$. Then the event $Z > Dn^{-1/2}$ is contained (for moderately large $D$) in the event that the rectangle $[x, x + Dn^{-1/2}/2]\times [y,1]$ has at most $K-1$ points from the set $\bx_{i+1}, \dots, \bx_n$. Since the locations of these points are uniform in the rectangle $[0,1] \times [y,1]$ (under our condition that $\xi_2(\bx_i) = y$), the probability of this event is exactly $\Pr(X < K)$ where $X$ is distributed as $\text{Binomial}(2n^{1/2},Dn^{-1/2}/2)$. Noting that $\E X = D$, the Chernoff bounds immediately imply \eqref{PZ} for large enough $D$, completing the proof. \qed

\section{Proof of Theorem \ref{EucTours}(a)}\label{sec:tours}

We start by revealing the positions of the vertices and setting aside a set $R$ of $r$ arbitrary vertices where $r=Cn^{1/2}$ for some constant $C$ large enough to make the following argument work. Set $N' = [n] \sm R$ and $n' =n - r$. We then run the following greedy algorithm on $N'$ for $k_0:=n'-r$ steps: we begin with $A=\emptyset$ and let $G_A=(N',A)$ be the graph induced by the selected edges $A$ (we will maintain that $G_A$ is a disjoint set of paths). We consider the natural greedy algorithm which iteratively adds the cheapest edge to $A$ that (i) does not create a cycle in $G_A$, (ii) does not create a vertex of degree 3 in $G_A$, and (iii) does not repeat a color. The algorithm is implemented as follows: we start with all edges available and the edges naturally ordered by length; at stage $i+1$, we choose the cheapest available edge and reveal its color $c$; if this color appears already in $G_A$, then we move to the next cheapest available edge and repeat the process until the revealed color does not appear in $G_A$; we add this edge to $G_A$; we then remove from the available edges any edge whose addition to $A$ would violate condition (i) or (ii) above.

After $k$ iterations, say there are $m$ nontrivial paths of lengths $\ell_1,\ell_2,\ldots,\ell_m\geq1$ where $\ell_1+\cdots+\ell_m=k$. Let $\bx_1,\ldots,\bx_{m}$ be made up from one endpoint of each path and let $\bx_{m+1},\ldots, \bx_{n'-k}$ be the isolated vertices (not in any edge of $A$). Let $X_k=\{\bx_1,\ldots,\bx_{n'-k}\}$.

We claim that with failure probability $o(1/n)$, there is an edge of $X_k$ of length at most $\l_k=Cn^{1/2}/(n'-k)$ whose color does not appear in $A$. To see this, partition $[0,1]^2$ into $m_k=2/\l_k^2$ sub-squares of side $\l_k/2^{1/2}$, and let $S$ be an arbitrary set of points of size $n' - k$. Let $s_i$ denote the number of points of $S$ in square $i$. Then, by convexity,
\[
\sum_{i=1}^{m_k}\binom{s_i}{2}\geq m_k\binom{|X_k|/m_k}{2}\geq C^2n/5.
\]
Then, the probability that there is a set $S$ of $n'-k$ points (from $N'$) and a set $Q$ of $q-k$ colors (from our palette of size $q:=(1 + \ep)n$) such that no sub-square contains a pair of points joined by an edge with color in $Q$ can be bounded by
\begin{align*}
\binom{n'}{n' - k}\binom{q}{q-k}\bfrac{k}{q}^{C^2n/5}&= \binom{n'}{k}\binom{q}{k}\bfrac{k}{q}^{C^2n/5}\\
&\leq \brac{\frac{ne}{k}\cdot\frac{qe}{k}\cdot\bfrac{k}{q}^{C^2n/(5k)}}^k\\
&\leq \brac{e^2(1 + \e)\bfrac{1}{1 + \e}^{C^2n/(5k)-2}}^k=o(n^{-1}),
\end{align*}
for sufficiently large $C=C(\e)$.

At the end of the greedy algorithm, there are exactly $r$ components of $G_A$ since at each step the number of components drops by one. We claim that we can complete $G_A$ to a rainbow Hamilton cycle using edges in a graph $\nabla(N',R)$, the edges with one end in $N'$ and one end in $R$, noting that the colors of all edges in $\nabla(N',R)$ remain unrevealed. Use $C_1$ for the set of colors appearing in $G_A$.

The only edges in $\nabla(N',R)$ that are relevant for completing $G_A$ to a Hamilton cycle are those incident to an endpoint of a (possibly trivial) path in $G_A$. Supposing that there are $p$ nontrivial paths in $G_A$, we think of the relevant edges of $\nabla(N',R)$ as a partially directed bipartite graph $B = (X,Y;E)$ where $|X| = |Y| = |R|$. $X$ is the union of $P$, one vertex for each non-trivial path, and $I$, one vertex for each isolated vertex. $Y$ is a copy of $R$ and $E$ is the union of the complete bipartite graph on $I \cup Y$ together with the complete directed graph on $P \cup Y$. Observe that a (partially directed) rainbow Hamilton cycle on $B$ using colors from $C \sm C_1$ together with the edges from $G_A$ forms a rainbow Hamilton cycle on $[n]$. Moreover, since $|R| \approx \sqrt{n}$, the cost of any Hamilton cycle on $B$ is $O(\sqrt{n})$. So to finish the proof, it suffices to find a rainbow Hamilton cycle on $B$ using $C \sm C_1$.

Reveal the colors on $B$, noting that an edge receives a color from $C \sm C_1$ with probability $p_1=\e/(1+\e)$. So let the edges of $\G_1=B_{p_1}$ be randomly colored using a palette of $q'\geq \e n$ colors. Note next that the expected number of times a particular color is used is at most $\l:=2r^2/q'=O(1)$. It is not difficult then to show using McDiarmid's inequality \cite{McD} that the set of colors $C_2$ used exactly once satisfies $|C_1|\gtrsim q'e^{-\l}$ with failure probability $o(1/n)$. Now consider the subgraph $\G_2$ consisting of edges in $\G_1$ whose colors are in $C_2$. Conditioning on $|\G_2|=m$, $\G_2$ will be distributed as (a supergraph of) $B_m$, a random set of $m$ edges from $B$. All we need to prove now is the following:
\begin{proposition}
With probability $1 - o(1/n)$, $\G_2$ contains a Hamilton cycle.
\end{proposition}
\begin{proof}
This follows from a result of Frieze \cite{FriezeHam} if $B$ contains no directed edges. Moreover, $\G_2$ can be coupled with $B_{p_2}$, the binomial random subgraph of $B$ with $p_2 \sim e^{-\l}q'/n$ with failure probability $o(1/n)$. To finish, note that McDiarmid's coupling \cite{Mc80} shows that $B_{p_2}$ is more likely to contain a Hamilton cycle than $K_{r,r,p_2}$ (which contains a Hamilton cycle with probability $1 - o(1/n)$ by \cite{FriezeHam}). 
\end{proof}
This completes the proof of Theorem \ref{EucTours}.

\section{Separating $Z_{\MST}$ from $Z_{\MST}^*$ and $Z_{\TSP}$ from $Z_{TSP}^*$ -- Theorem \ref{EucTours}(b)}\label{sec:separate}
We can show the separations
\begin{align}
  \label{zmsts}Z_{\MST}-Z_{\MST}^*=\Omega(\sqrt n)\quad\text{w.h.p}\\
  \label{ztsps}Z_{\TSP}-Z_{\TSP}^*=\Omega(\sqrt n)\quad\text{w.h.p}.
\end{align}
using the methods of \cite{FP1} (which there were used to show, among other things, that $Z_{TSP}^*-Z_{MST}^*=\Omega(\sqrt n)$).

We begin with a simple lemma.
\begin{lemma}\label{repeat}
Let $S$ be a set of size $\a=\Theta(n)$ that is randomly colored using $\b=\Theta(n)$ colors. Then w.h.p. there are $\Theta(n)$ colors that are used more than once.
\end{lemma}
\begin{proof}
Let $Z$ denote the number of colors that appear more than once. Then 
\[
\E(Z)=\b\brac{1-\brac{1-\frac{1}{\b}}^{\a}-\a \cdot \frac{1}{\b}\brac{1-\frac{1}{\b}}^{\a-1}}=\Theta(n).
\]
Now changing the color of one edge changes $Z$ by at most one. So, applying McDiarmid's inequality \cite{McD} we have 
\[
\Pr\brac{Z\leq \E(Z)/2}\leq \exp\set{-\frac{\E(Z)^2}{2\a}}=e^{-\Omega(n)}.
\]
\end{proof}
Now back to the main argument. For notational convenience, in this section we scale $\cX$ so that we are instead working with a set $\cY_n$ of $n$ random points in a square of side length $\sqrt{n}$; note that there is thus one point on average, per unit of area.

We first recall a definition and lemma used in \cite{FP1}.  Given $\e,D>0$ and a finite set of points $S\subseteq \mathbb{R}^d$ and a universe $Y$, we say that $T\subseteq Y$ is an $(\e,D)$-copy of $S$ if there is a bijection $f$ between $T$ 
and $S$ 
such that $||{x}-{f(x)}||<\e$ for all $x\in T$, and such that $T$ is at distance $>D$ from $Y\setminus T$.  First we have:
\begin{proposition}[see Observation 2.1 in \cite{FP1}] \label{smiley}
Given any finite point set $S$, any $\e>0$, and any $D$, $\cY_n$ w.h.p contains at least $C^S_{\e,D}n$ $(\e,D)$-copies of $S$, for some constant $C^S_{\e,D}>0$.
\end{proposition}
  
This already suffices to show \eqref{zmsts}.  Indeed, apply Proposition \ref{smiley} to a set $S$ consisting of two points at unit distance, with $\ep=1/4$ and $D=4$.  We learn that w.h.p. the number $\kappa$ of $(\frac 1 4,4)$-copies of $S$ in $\cY_n$ satisfies $\kappa\geq Cn$ for a constant $C$.  Since the colors of edges are independent of the geometry of the points, we can apply Lemma \ref{repeat} with $\a=\k,\b=q$ to show that w.h.p. there are $\Omega(\kappa)$ colors that are repeated among the $\kappa$ edges which join the two points in a given copy.  Thus for any rainbow MST, there are at least $\Omega(n)$ copies in which it does not use the edge between the two points.  We can add this edge and delete one of the edges leaving the copy to produce a spanning tree shorter by at least unit distance.  Doing this for each of these $\Omega(n)$ copies produces a spanning tree shorter than the rainbow MST by $\Omega(n)$.  Upon rescaling to the unit square, this corresponds to the existence of a spanning tree of length $\Omega(\sqrt{n})$ less than the length of the rainbow MST.

To prove \eqref{ztsps}, we will additionally use the following results from \cite{FP1}:
\begin{proposition}[Observation 2.9 in \cite{FP1}]\label{2meansstraight}
Suppose that $S_{\e,D}$ is an $(\e,D)$ copy of a fixed set $S$ for fixed $\e$ and sufficiently large $D$, and that at least 2 pairs of edges of a shortest TSP tour $\cL$ join $S_{\e,D}$ to $V\setminus S_{\e,D}$.  Then the pairs are nearly straight (i.e., the angle for each pair is arbitrarily close to $\pi$ as $\e\to 0$, and $k,D\to \infty$).  Moreover, any tour without this property can be shortened by a length bounded below by a function just of $\ep$, $S$, and $D$ by changing just how it passes through $S_{\ep,D}$.
\end{proposition}

\begin{proposition}[see Observation 2.10 in \cite{FP1}]
  \label{nothree}
Suppose that $S_{\e,D}$ is an $(\e,D)$-copy of any fixed set $S$ for fixed $\e$ and sufficiently large $D$.  Then there are at most 2 pairs of edges in a shortest TSP tour which join $S_{\e,D}$ to $V\setminus S_{\e,D}$.  Moreover, any tour with at least 3 such pairs can be shortened by a length bounded below by a function just of $\ep$, $S$, and $D$ by changing just how it passes through $S_{\ep,D}$.
\end{proposition}

\begin{figure}[t]
\begin{center}
\begin{pdfpic}
\psset{unit=1cm,dotsize=3pt}
\begin{pspicture}(-1.5,-1)(2.5,2.866)
\SpecialCoor
\rput{0}(1,0){
\psdot(.1,0)
\psdot(-.1,0)
}

\rput{0}(-1,0){
\psdot(.1,0)
\psdot(-.1,0)
}
\rput{0}(0,1.73){
\psdot(.1,0)
\psdot(-.1,0)
}

\psline(-1.5,-.866)(-.9,0)(-1.1,0)(-1,.173)(.1,1.73)(-.1,1.73)(0,1.903)(.5,2.598)

\psline(.5,-.866)(.9,0)(1.1,0)
(2.5,2.5)

\end{pspicture}
\end{pdfpic}
\end{center}
\caption{\label{f.S} An instance of the set $S$ used, together with an example of how an (uncolored) optimum tour may traverse this set in a way consistent with Propositions \ref{nothree} and \ref{2meansstraight}.}
\end{figure}

We now let $v_1,v_2,v_3$ denote the vertices of a fixed unit equilateral triangle, and choose a large $D$, small $\delta$ and even smaller $\ep$, so that $1\gg\delta\gg \ep>0$, and apply \ref{smiley} to a set $S$ consisting of three pairs of distinct points $p_i,q_i$, $i=1,2,3$, where the points $p_i,q_i$ are both within distance $\delta$ of $v_i$ (Figure \ref{f.S}).  In particular, we learn that there are $\kappa\geq Cn$ $(\ep,D)$-copies of this set $S$ in $\cY_n$.  

Propositions \ref{2meansstraight} (applied both to $S$ and to the individual pairs $\{p_i,q_i\}$ separately) and \ref{nothree} (applied to $S$) now imply that for each such copy, any rainbow tour through $\cY_n$ tour either:
\begin{enumerate}[(a)]
\item Traverses the copy that in a way that uses all of the edges $\{p_i,q_i\}$, $i=1,2,3$, or else
\item \label{gamma} Traverses the copy in a way that can be improved by at least a fixed distance $\gamma>0$ by only changing the path(s) the tour takes through the copy.
\end{enumerate}

Our goal is to show that there are a linear number of copies for which the optimum rainbow tour falls into case \ref{gamma}; this then implies that in the (rescaled) unit square, the optimum rainbow tour is longer by an additive factor of $\Omega(\sqrt{n})$.  Note that overall, there are $3\kappa\geq 3Cn$ edges $\{p_i,q_i\}$ over all $i=1,2,3$ and all $(\ep,D)$-copies of $S$.  The colors of edges are independent of the geometry of the points. Applying Lemma \ref{repeat} with $\a=3\k,\b=q$ we see that w.h.p. there are $\Omega(n)$ colors that are repeated in the $(\ep,D)$-copies of $S$. Thus w.h.p. there are $\Omega(n)$ instances of case (\ref{gamma}) with respect to the rainbow tour, as desired.

\section{Proof of Theorem \ref{th2}}\label{sec:spread}
Here we assume that the costs $C(i,j)$ are independent uniform $[0,1]$. The general idea is to replace an edge $(i,j)$ of color $c$ with an edge $\set{i,c,j}$ in a random 3-uniform hypergraph. There is a technical problem in that we are not allowed to have $\set{i,c,j},\set{i,c',j}$ for $c\neq c'$. We deal with this by merging the arguments of Frankston, Kahn, Narayanan and Park \cite{FKNP} and Bell, Frieze and Marbach \cite{BFM}. 

Let $X$ be a finite set and let $\cH$ be an $r$-bounded hypergraph on $X$. For $k \geq r$, let $X^* = X \times [k]$. Let $\cH^*$ denote the family consisting of all possible rainbow copies of $\cH$ using the color set $[k]$.  Let $\xi_x\,(x\in X)$ be independent random variables, each uniform from $[0,1]$, and set
\[
\xi_{\cH}=\min_{S\in \cH^*}\sum_{x\in S}\xi_x.
\]
We prove
\begin{theorem}\label{rweight}
Let $\cH$ be a $\k$-spread, $r$-uniform hypergraph on vertex set $X$. Suppose that (i) $\k=\Omega(r)$ and (ii) $|X|\leq \k^2r/\log^5r$. If we randomly color the elements of $\cH$ with $k\geq r$ colors, then $\xi_\cH\leq 3C_0r/\k$ w.h.p. (assuming $r\to\infty$) for some absolute constant $C_0>0$.
\end{theorem}
(The paper \cite{HY} removes conditions (i) and (ii), but the proof doesn't adapt well to deal with the costs, as far as we can see.)
\begin{proof}
We let $\ell_0$ be the least $i$ such that $r/2^i<\log r$, and let $\ell$ be the least $i>\ell_0$ such that $r/2^i<1$. We sort the elements of $X=\set{x_1,x_2,\ldots,x_N}$ so that $\xi_{x_1}<\xi_{x_2}<\cdots<\xi_{x_N}$. We then let $p_i=C_i/\k$ where 
\[
C_i=\begin{cases}C_0&i\leq \ell_0.\\\frac{\log r}{\log\log r}&\ell_0<i\leq \ell.\end{cases}
\]
Here $C_0$ is a large constant. Let $L_i=Np_i$ and define $n_0$ to be the smallest integer such that $2^{-n_0}\leq \log r$ and let $n$ be the least $i>n_0$ such that $r/2^i<1$. Then for $i\in[n]$ we let 
\[
a_i=\sum_{j=1}^iL_i\text{ and }\e_i=\frac{2a_i}{N}.
\]
Then let $W^*_1=\set{x_1,x_2,\ldots,x_{a_1}},\,W^*_2=\set{x_{a_1+1},x_{a_1+2},\ldots,x_{a_2}}$ etc. Note that $W^*_1,W^*_2,\ldots,W^*_n$ are randomly colored random sets of size $L_i$ . The Chernoff bounds imply that
\begin{align}
\Pr(\exists i:\max\set{\xi_x:x\in W^*_i}\geq \e_i)&=\sum_{i=1}^n\Pr(Bin\brac{N,\e_i}\leq\e_i N)\nn\\
& \leq \sum_{i=1}^n e^{-\e_iN/8}\leq 2e^{-a_1/4}=2e^{-C_0N/\k}.\label{max}
\end{align}
In \cite{BFM} (as we will explain shortly), it is proved that with probability $1-\frac{\log\log r}{\log^{1/4}r}$, there is a rainbow colored edge $e$ such that $|e\setminus \bigcup_{j\leq i}W^*_j|\leq r/2^{i-1}$. We will use the main theorem (and its proof which we discuss below) from \cite{BFM}:
\begin{theorem}[Theorem 2 in \cite{BFM}]\label{thrainbow}
Let $\cH$ be an $r$-bounded, $\k$-spread hypergraph and let $X=V(\cH)$ be randomly colored from $Q=[q]$ where $q\ge r$. Suppose also that (i) $\k=\Omega(r)$, that is, there exists a constant $L>0$ such that $\kappa\ge Lr$ for all valid $r$, and that (ii) $N\leq \k^2r/\log^5r$. Then there is a constant $C>0$ such that if 
\beq{mbounds}{
m\geq\frac{(C\log_2r)|X|}{\k}
}
then $X_m$ contains a rainbow colored edge of $\cH$ w.h.p. 
\end{theorem}
 The proof of Theorem \ref{thrainbow} in \cite{BFM} proceeds by iteratively choosing random sets of colored elements $W^*_i$ (just as we have in our current proof). Theorem \ref{thrainbow} is proved by iteratively applying the lemma we state next. Here $\cH^*$ is the set of possible rainbow colored edges of $\cH$. For $H^* \in \cH^*$, $T^*(W^*,H^*)$ is $G^* \setminus W^*$ for some $G^* \in \cH^*$.
 \begin{lemma}[Lemma 5 in \cite{BFM}] \label{ML}
Let $H^*\in\cH^*$ be good with respect to $W^*$ if $H^*\sim W^*$ and $|T^*(W^*,H^*)|<r/2$. Let 
\[
\text{$\cS$ be the event $|\set{H^*\in \cH^*:H^*\text{ is good}}|<(1-\e)|\cH^*|(1-p)^r$.}
\]
Then
\[
\Pr(\cS)\leq \exp\set{-\frac{\e^2\k^2q(1-p)}{16e^3N}} +\frac{2K}{C_0^{r/3}\e\brac{1-p}^r}
\]
\end{lemma}
In other words, so long as the unlikely event $\cS$ does not occur, we have that for each rainbow colored edge $H^*$, there is some rainbow colored edge $G^*$ that is mostly contained in $W^*$. The proof of Theorem \ref{thrainbow} proceeds by iteratively applying Lemma \ref{ML}, each time replacing $\cH^*$ with the hypergraph whose edges are the $T^*(W^*, H^*)$ (i.e. portions of edges of $\cH^*$ which are not contained in $W^*$). In \cite{BFM} they show that during all applications of the Lemma \ref{ML}, 
 the unlikely event $\cS$ never happens w.h.p. This means that the rainbow edge $H^*$ we obtain in the end was good for every application of Lemma \ref{ML}. In particular, we have
    \beq{forcost}{
|H^*\setminus \bigcup_{j\leq i}W^*_j|\leq \frac{r}{2^{i-1}}, \qquad \text{i=1, \ldots, $\ell$.}
}
In our setting, since $W^*_i$ was the set of vertices of cost at most $\e_i$, the rainbow edge $H^*$ we obtain from Theorem \ref{thrainbow} costs at most
\mults{
\sum_{i=1}^n\frac{r\e_i}{2^{i-1}}=\frac{r}{N}\sum_{i=1}^n\frac{1}{2^{i-1}}\sum_{j=1}^iNp_i\\
=r\brac{\sum_{i=1}^{\ell_0}\frac{C_i}{2^{i-1}\k}+\sum_{i=\ell_0+1}^\ell\frac{C_{\ell_0}+(i-\ell_0)\log r}{2^{i-1}\k}}\leq \frac{3C_0r}{\k}=O(1).
}

\end{proof}
All of the examples given in Theorem \ref{th2} have the requisite spread to justify the claim. For (1) -- (3) we have $r=O(n)$ and $\k=\Omega(n^{k-1})$. (For (3), we have $k=2$.) 

This concludes the proof of Theorem \ref{th2}.
\section{Observations and open questions}\label{sec:conclusion}
We have estimated the effect of adding a rainbow constraint in the probabilistic analysis of some random combinatorial optimization problems. For the Euclidean case we are tight up to a constant if we are dealing with the minimum spanning tree problem. For the TSP and $(1 + \e)n$ colors, there is a factor of $\log n$ between our lower and upper bounds. It is an open question as to whether this is actually necessary. 

Karp \cite{K2}  gave a polynomial time algorithm that is asymptotically optimal for the TSP in the unit square. It would be interesting to see if there is an analogous algorithm for the Rainbow TSP in the unit square. One also wonders if the algorithms of Arora \cite{Arora} or Mitchell \cite{Mit} can be adapted to the same problem.

Can we find the correct constants for any of the problems mentioned in Theorem \ref{th2}?

\end{document}